\documentclass[12pt]{amsart}
\usepackage{amsmath}
\usepackage{amscd}
\usepackage{amssymb}
\usepackage{amsfonts}
\usepackage{amsthm}
\usepackage{bbm}
\usepackage{cancel}
\usepackage{color}
\usepackage{eucal}
\usepackage{enumerate,yfonts}
\usepackage{enumitem}
\usepackage{relsize}
\usepackage{float}

 \allowdisplaybreaks

\usepackage{pdfsync}
 \usepackage[all,cmtip]{xy}
\usepackage{graphicx}
\usepackage{graphics}
\usepackage{hyperref}
\usepackage{latexsym}
\usepackage{mathrsfs}
 \usepackage{placeins}
\usepackage{pstricks}
\usepackage{bookmark}

\usepackage{stmaryrd}
\usepackage{url}

\DeclareFontFamily{U}{mathx}{\hyphenchar\font45}
\DeclareFontShape{U}{mathx}{m}{n}{
      <5> <6> <7> <8> <9> <10>
     <10.95> <12> <14.4> <17.28> <20.74> <24.88>
    mathx10
      }{}
\DeclareSymbolFont{mathx}{U}{mathx}{m}{n}
\DeclareMathAccent{\widecheck}{\mathalpha}{mathx}{"71}

\newtheorem{thm}{Theorem}[section]
\newtheorem{corollary}[thm]{Corollary}
\newtheorem{lemma}[thm]{Lemma}

\newtheorem{proposition}[thm]{Proposition}

\newtheorem{conjecture}[thm]{Conjecture}
\newtheorem{thm-dfn}[thm]{Theorem-Definition}

\newtheorem{remark}[thm]{Remark}

\newtheorem{theorem}{Theorem}[section]

\numberwithin{equation}{section}

\newcommand{\fp}{{\mathfrak p}}

\newcommand{\fa}{{\mathfrak a}}

\newcommand{\fF}{{\mathfrak{F}}}

\newcommand{\Lp}{{\mathfrak{p}}}

\newcommand{\Lb}{{\mathfrak{b}}}
\newcommand{\Lt}{{\mathfrak{t}}}
\newcommand{\Lg}{{\mathfrak g}}

\newcommand{\Ln}{{\mathfrak{n}}}

\newcommand{\Ll}{{\mathfrak{l}}}

\newcommand{\rf}{{\mathrm f}}
\newcommand{\rn}{{\mathrm n}}

\newcommand{\bC}{{\mathbb C}}

\newcommand{\bZ}{{\mathbb Z}}

\newcommand{\bP}{{\mathbb P}}

\newcommand{\bN}{{\mathbb N}}

\newcommand{\calF}{{\mathcal F}}

\newcommand{\cO}{{\mathcal O}}
\newcommand{\cA}{{\mathcal A}}
\newcommand{\cF}{{\mathcal F}}
\newcommand{\cN}{{\mathcal N}}
\newcommand{\cH}{{\mathcal H}}

\newcommand{\cT}{{\mathcal T}}

\newcommand{\cP}{{\mathcal P}}
\newcommand{\cE}{{\mathcal E}}

\newcommand{\cM}{{\mathcal{M}}}

\newcommand{\cL}{{\mathcal{L}}}

\newcommand{\on}{\operatorname}

\newcommand{\Loc}{\on{LocSys}}

\newcommand{\nc}{\newcommand}

\nc{\al}{{\alpha}} \nc{\be}{{\beta}} \nc{\ga}{{\gamma}}
\nc{\ve}{{\varepsilon}} \nc{\Ga}{{\Gamma}} 
\nc{\La}{{\fa}}

\nc{\ad }{{\on{ad }}}

\nc{\aff}{{\on{aff}}} \nc{\Aff}{{\mathbf{Aff}}}

\nc{\der}{{\on{der}}}

\nc{\diag}{{\on{diag}}}

\nc{\Fl}{{\calF\ell}}

\nc{\Hg}{{\on{Higgs}}}

\nc{\Id}{{\on{Id}}}

\nc{\Ind}{{\on{Ind}}}

\nc{\Op}{{\on{Op}}}

\nc{\res}{{\on{res}}}

\nc{\tr}{{\on{tr}}}

\nc{\GSp}{{\on{GSp}}} \nc{\GU}{{\on{GU}}} \nc{\SL}{{\on{SL}}}
\nc{\SU}{{\on{SU}}} \nc{\SO}{{\on{SO}}}

\nc{\nh}{{\Loc_{J^p}(\tau')}}
\nc{\bnh}{{\Loc_{\breve J^p}(\tau')}}

\nc{\bU}{{\overline{U}}} 
\nc{\IC}{{\on{IC}}}

\nc{\op}{{\operatorname{P}}}

\newcommand{\br}{\begin{rouge}}
\newcommand{\er}{\end{rouge}}

\newcommand{\bb}{\begin{bluet}}
\newcommand{\eb}{\end{bluet}}

\newcommand{\bfd}{{\mathbf d}}

\nc{\ot}{\otimes}

\nc{\oh}{{\operatorname{H}}}
\nc{\gr}{{\operatorname{gr}}}
\nc{\rk}{{\operatorname{rank}}}
\nc{\codim}{{\operatorname{codim}}}
\nc{\img}{{\operatorname{Im}}}
\nc{\Span}{{\operatorname{Span}}}
\nc{\Img}{\operatorname{Im}}
\nc{\Char}{\operatorname{Char}}

\newcommand{\beqn}{\begin{equation*}}
\newcommand{\eeqn}{\end{equation*}}

\newcommand{\beq}{\begin{equation}}
\newcommand{\eeq}{\end{equation}}

\newcommand{\bega}{\begin{gathered}}
\newcommand{\eega}{\end{gathered}}

\newcommand{\bern}{\begin{eqnarray*}}
\newcommand{\eern}{\end{eqnarray*}}

\newcommand{\ber}{\begin{eqnarray}}
\newcommand{\eer}{\end{eqnarray}}

\nc{\nmod}{\hspace{-.1in}\mod}

\begin{document}
\title[Character sheaves for classical graded  Lie algebras]{Character sheaves for classical graded  Lie algebras}
       
         \author{Ting Xue}
         \address{ School of Mathematics and Statistics, University of Melbourne, VIC 3010, Australia, and Department of Mathematics and Statistics, University of Helsinki, Helsinki, 00014, Finland}
         \email{ting.xue@unimelb.edu.au}
\thanks{Ting Xue was supported in part by the ARC grants DP150103525.}

\begin{abstract}In this note we study character sheaves for graded Lie algebras arising from inner automorphisms of special linear groups and Vinberg's type II classical graded Lie algebras.
\end{abstract}

\maketitle

\section{Introduction}In~\cite{VX} we initiated a study of character sheaves for $\bZ/m\bZ$-graded Lie algebras. The invariant theory of graded Lie algebras was studied by Vinberg in~\cite{V} where the classical graded Lie algebras were divided into types I, II and III. In~\cite{VX,VX2} we focus on type I classical graded Lie algebras (and some exceptional cases), which is the most complicated type in the following sense. There are families of cuspidal character sheaves, with full support, associated to irreducible representations of various Hecke algebras of complex reflections groups at roots of unity.  In this note we study character sheaves in the cases of gradings arising from inner automorphisms of special linear groups, which we refer to as type AI, and Vinberg's type II classical graded Lie algebras. The latter can be viewed as the simplest type, in the sense that we do not expect the existence of cuspidal character sheaves. For type III, we expect that cuspidal character sheaves are rare and further that they all have nilpotent support, as in Lusztig's generalised Springer correspondence~\cite{L1}. We will deal with type III and make connection (in all types) to the Lusztig-Yun work~\cite{LY} in future publications.

Let us briefly recall the set-up. Let $G$ be a reductive algebraic group and $\theta:G\to G$ an order $m$ automorphism of $G$. We have a decomposition $\Lg=\oplus_{i\in\bZ/m\bZ}\Lg_i$  of $\Lg=\on{Lie}G$ into eigenspaces of $d\theta$, where $d\theta|_{\Lg_i}=\zeta_m^i$ and we write $\zeta_k=e^{2\pi\mathbf{i}/k}$ for $k\in\bZ_{+}$. Let $K=(G^\theta)^0$. By Vinberg~\cite{V}, the invariant theory of $K$-action on $\Lg_1$ is analogous to that of the adjoint action of $G$ on $\Lg$. In particular, there exists a Cartan subspace $\fa\subset\Lg_1$ consisting of semisimple elements such that $\bC[\Lg_1]^K\cong\bC[\fa]^{W_\fa}$, where $W_\fa\cong N_K(\fa)/Z_K(\fa)$ is the little Weyl group. The group $W_\fa$ is in general a complex reflection group. We are interested in describing explicitly the simple $K$-equivariant perverse sheaves on $\Lg_1$ that are Fourier transforms of simple $K$-equivariant perverse sheaves on $\cN_{-1}=\Lg_{-1}\cap\cN$, where $\cN$ is the nilpotent cone of $\Lg$. Here Fourier transform refers to the functor $\fF:\on{Perv}_K(\Lg_{-1})\to\on{Perv}_K(\Lg_{1})$, where we have identified $\Lg_1$ with $\Lg_{-1}^*$. We call these perverse sheaves character sheaves for graded Lie algebras, and write $\on{Char}_K(\Lg_1)$ for them. Let $\cA_K(\cN_{-1})$ denote the set of simple $K$-equivariant perverse sheaves on $\cN_{-1}$, called nilpotent orbital complexes. We have $\on{Char}_K(\Lg_1)=\fF(\cA_K(\cN_{-1}))$ by definition. A character sheaf is called cuspidal if it does not arise as a direct summand (up to degree shift) from parabolic induction of character sheaves in $\on{Char}_{(L^\theta)^0}(\Ll_1)$ for any $\theta$-stable Levi subgroup $L$ contained in a $\theta$-stable parabolic subgroup. We write $\on{Char}_K^{cusp}(\Lg_1)$ for them. We also write $\on{Char}_K^{\rn}(\Lg_1)$ (resp. $\on{Char}_K^{\rf}(\Lg_1)$) for the set of nilpotent support (resp. full support) character sheaves.

In this note we describe the character sheaves explicitly (Theorem~\ref{cs-sl} for type AI and Theorem~\ref{cs} for type II) under the assumption that the conjectural description (Conjecture~\ref{conj-nilp} and Conjecture~\ref{conj-biorbital}) of the set $\on{Char}_K^{\rn}(\Lg_1)$ of nilpotent support character sheaves holds. We also give a conjectural description of cuspidal character sheaves for type AI (Conjecture~\ref{conj-cusp}). Recall that we expect that there are no cuspidal character sheaves in type II. In Section~\ref{sec-pre} we recall the graded Lie algebras considered here and write down the quiver description of the $K$-action on $\Lg_1$ following~\cite{V,Y}. We also recall some braid group representations that will be used to describe the character sheaves. In Section~\ref{sec-nilp} we parametrize the nilpotent orbits and determine their equivariant fundamental groups. In particular, we determine the distinguished nilpotent orbits. We also write down the total number of nilpotent orbits (Lemma~\ref{lem-nb1}) and that of distinguished orbits (Lemma~\ref{lem-nb2}). We discuss the dual strata which give rise to supports of character sheaves. In Section~\ref{sec-cs-sl} we deal with inner automorphisms of special linear groups. As in~\cite{VX3} we make use of a generalised nearby cycle construction and parabolic induction to produce character sheaves. We note that the character sheaves with trivial central character have been described in~\cite{L}. In Section~\ref{cs-typeII} we discuss type II. In particular, we expect that all nilpotent support character sheaves arise from parabolic induction of $\theta$-stable Borel subgroups. 

{\bf Acknowledgement.} I  thank Pavel Etingof, Oscar Kivinen, Ivan Losev, George Lusztig, Emily Norton,  Peng Shan, Cheng-Chiang Tsai, Kari Vilonen, and Zhiwei Yun for helpful discussions. Special thanks are due to Dennis Stanton for teaching me the counting arguments.

\section{Preliminaries}\label{sec-pre}

\subsection{Classical graded  Lie algebras}\label{ssec-gla}

We recall  the explicit description of gradings arising from inner automorphisms of special linear groups and type II classical graded Lie algebras,  following  Vinberg \cite{V}. Let $\theta:G\to G$ be an order $m$ automorphism. Recall $K=(G^\theta)^0$ and $\zeta_k=e^{2\pi\mathbf{i}/k}$ for $k\in\bZ_{+}$.   

{\bf Type AI.} Let $G=SL_V$, where $V$ is a complex vector space of dimension $N$. We can assume that $\theta(g)=\gamma g\gamma^{-1}$, where $\gamma\in G$ and $\gamma^m=1$. 

{\bf Type AII.} Let $G=SL_V$, where $V$ is a complex vector space of dimension $2n$ equipped with a non-degenerate bilinear form $(-,-)$ defined by $(v,w)=v^tAw$, $v,w\in V$. We can assume that $ \theta(g)=A^{-1}(g^t)^{-1}A$, $g\in G$. We have $\theta^2(g)=\gamma g\gamma^{-1}$, where $\gamma=A^{-1}A^t$. Then $m=2m_0$ for some odd $m_0\in\bZ_{+}$ and we can assume that $\gamma^{m_0}=-1$.

{\bf Type CII (resp. DII).} Let $G=Sp_V$ (resp. $SO_V$), where $V$ is a complex vector space of dimension $2n$  equipped with a non-degenerate symplectic (resp. symmetric bilinear) form $(\,,\,)$.  We can assume that 
$
\theta(g)=\gamma g\gamma^{-1},\,\gamma\in Sp_V \text{ (resp. $O_V$)}.
$
We have that $m$ is even and $\gamma^{m}=1$ (resp. $-1$).

In each case we have $V=\oplus V_\lambda$, where $V_\lambda:=\{v\in V\mid \gamma v=\lambda v\}$, $\lambda\in\bC$. 
Moreover,
\bern
&\Lg_k=\{x\in\Lg\mid xV_\lambda\subset V_{\zeta_{m_0}^k\lambda},\ (xv,w)+\zeta_{m}^k(v,xw)=0,\ \forall v,w\in V\}&\text{AII}
\\
&\Lg_k=\{x\in\Lg\mid xV_\lambda\subset V_{\zeta_{m}^k\lambda}\}&\text{AI,\,CDII}\\
&W_\fa=G_{m_0,1,r} \text{ (resp. $G_{m,1,r}$)}&\text{AII (resp. AI, CDII)}.
\eern
Let 
\bern
l=(m_0-1)/2\text{ (resp. $l=m/2$)}&&\text{AII (resp. CDII)}\\
\text{ and $M_i=V_{\xi_{m}^{-i}}$ (resp. $V_{\xi_{m}^{1-2i}}$, resp. $V_{\xi_{2m}^{1-2i}}$)}&&\text{AI, CII (resp. AII, resp. DII)}.
\eern
We have
\beqn
r=\on{dim}\fa=\on{min}\{\on{dim}M_i\}\text{ (resp. $\min\{[\frac{\on{dim}M_i}{2}]\}$)}\ \ \text{ AI (resp. ACDII)}.
\eeqn

We have the following description of the pair $(K,\Lg_1)$ using quivers, following~\cite{Y} (see also~\cite{YY}). We write $\mathbf{d}=(d_i)$, $d_i=\on{dim}M_i$, for the dimension vector of the quiver, and $|\mathbf{d}|=\sum d_i$.

{\bf Type AI.} We can identify $\Lg_1$ with the representations of the following cyclic quiver
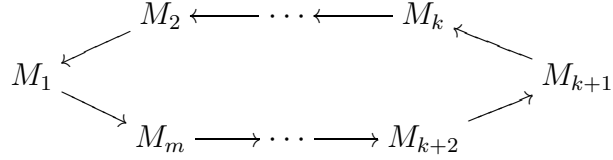
\beqn
\xymatrix@R=2mm{&M_2\ar[dl]&\cdots\ar[l]&M_{k}\ar[l]&\\M_1\ar[dr]&&&&M_{k+1}\ar[ul]\\&M_m\ar[r]&\cdots\ar[r]&M_{k+2}\ar[ur]&}
\eeqn
We have
\bern
K\cong\prod_{j=1}^m GL_{M_j},\ \Lg_{1}\cong\bigoplus_{i=1}^m \on{Hom}(M_i,M_{i-1}).
\eern

{\bf Type AII.} We have  $(M_i,M_j)=0$ unless $i+j\equiv 1\mod m_0$, and $(,)|_{M_{l+1}}$ is a non-degenerate symplectic form. In particular, $d_{l+1}$ is even. We can identify $\Lg_1$ with the set of representations of the following cyclic quiver
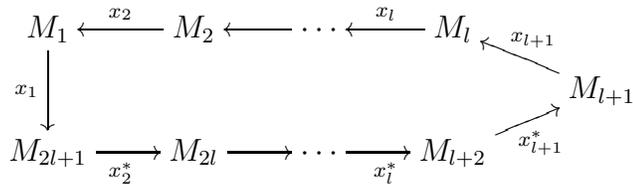
\beqn
\xymatrix@R=2mm{&M_1\ar[dd]_{x_1}&M_2\ar[l]_{x_2}&\cdots\ar[l]&M_{l}\ar[l]_{x_{l}}\\&&&&&M_{l+1}\ar[ul]_-{x_{l+1}}\\&M_{2l+1}\ar[r]_{x_2^*}&M_{2l}\ar[r]&\cdots\ar[r]_{x_{l}^*}&M_{l+2}\ar[ur]_{x_{l+1}^*}}
\eeqn
such that $(x_iv_i,w_{m_0+2-i})+\zeta_m(v_i,x_i^*w_{m_0+2-i})=0$, $i\in[2,l+1]$, and $(x_1v_{1},w_{1})=-(x_1w_{1},v_{1})$, for all $v_j,w_j\in M_j$.
We have
\beqn
\begin{gathered}
K\cong \prod_{j=1}^lGL_{M_j}\times Sp_{M_{l+1}},\ 
\Lg_1\cong\bigoplus_{i=1}^l\on{Hom}(M_{i+1},M_{i})\oplus\wedge^2(M_{1}^*).
\end{gathered}
\eeqn
{\bf Type CII.} We have  $( M_i,M_j)=0$ unless $i+j\equiv 0\mod m$, $(\,,\,)|_{M_l}$ and $(\,,\,)|_{M_m}$ are both non-degenerate.  In particular both $d_l$ and  $d_m$ are even. We can identify $\Lg_1$ with the set of representations of the following cyclic quiver
\beqn
\xymatrix@R=2mm{&M_1\ar[dl]_{x_1}&M_2\ar[l]_{x_2}&\cdots\ar[l]&M_{l-1}\ar[l]_{x_{l-1}}\\M_{2l}\ar[dr]_{x_1^*}&&&&&M_l\ar[ul]_-{x_l}\\&M_{2l-1}\ar[r]_{x_2^*}&M_{2l-2}\ar[r]&\cdots\ar[r]_{x_{l-1}^*}&M_{l+1}\ar[ur]_{x_l^*}}
\eeqn
such that $(x_iv_i,w_{m+1-i})+(v_i,x_i^*w_{m+1-i})=0$, $i\in[1,l]$, for all $v_j,w_j\in M_j$.
We have
\beqn
\begin{gathered}
K\cong \prod_{j=1}^{l-1}GL_{M_j}\times Sp_{M_l}\times Sp_{M_m},\ 
\Lg_1\cong\bigoplus_{j=1}^l\on{Hom}(M_j,M_{j-1}).
\end{gathered}
\eeqn
{\bf Type DII.} We have  $(M_i,M_j)=0$ unless $i+j\equiv 1\mod m$.
We can identify $\Lg_1$ with the set of representations of the following cyclic quiver
\beqn
\xymatrix@R=4.5mm{&M_1\ar[d]_{x_1}&M_2\ar[l]_{x_2}&\cdots\ar[l]&M_l\ar[l]_{x_l}\\&M_{2l}\ar[r]_{x_2^*}&M_{2l-1}\ar[r]&\cdots\ar[r]_{x_l^*}&M_{l+1}\ar[u]_{x_{l+1}}}
\eeqn
such that $(x_iv_i,w_{m+2-i})+(v_i,x_i^*w_{m+2-i})=0$, $i\in[2,l]$, and $(x_1v_{1},w_{1})=-(x_1w_{1},v_{1})$, and $(x_{l+1}v_{l+1},w_{l+1})=-(x_{l+1}w_{l+1},v_{l+1})$, for all $v_j,w_j\in M_j$. We have
\beqn
\begin{gathered}
K\cong \prod_{j=1}^lGL_{M_j},\ \ 
\Lg_1\cong\bigoplus_{j=2}^{l}\on{Hom}(M_j,M_{j-1})\oplus \wedge^2(M_1^*)\oplus \wedge^2(M_l).
\end{gathered}
\eeqn
\subsection{Hecke algebras associated to complex reflection groups}\label{Hecke}We make use of the notations in~\cite[\S7.1]{VX}. Let $\iota|m$. Let $\cH^\iota_{{G_{m,1,k}}}$ denote the Hecke algebra associated to the complex reflection group $G_{m,1,k}$ defined as the quotient of the group algebra $\bC[B_{G_{m,1,k}}]$ by the ideal generated by the elements (see~\cite{BMR})
\beqn
\sigma_{H_{s_{ij}^{(p)}}}^2-1,\,1\leq i<j\leq k,\,0\leq p\leq m-1,\ \ (\sigma_{H_{\tau_i}}^\iota-1)^{m/\iota},\,1\leq i\leq k.
\eeqn
By~\cite[Section 3]{M}, the irreducible representations of $\cH_{G_{m,1,k}}^\iota$ are parametrized by $\iota$-partitions of $k$, that is, $\iota$-tuples of partitions $(\nu^1,\ldots,\nu^\iota)$ such that $\sum_{i=1}^\iota|\nu^i|=k$. We write $\cP_\iota(k)$ for the set of $\iota$-partitions of $k$, and write $\cP(k)=\cP_1(k)$. For each $\tau\in\cP_\iota(k)$, let $L_\tau$ denote the irreducible representation of $\cH_{G_{m,1,k}}^\iota$ corresponding to $\tau$. We will also write $L_\tau$ for the irreducible representation of $B_{G_{m,1,k}}$ obtained by pulling back $L_\tau$ via $\bC[B_{G_{m,1,k}}]\to\cH^l_{G_{m,1,k}}$.

\section{Nilpotent orbits and component groups}\label{sec-nilp}

In this section we describe the parametrization of $K$-orbits in $\cN_1$ (or equivalently, $\cN_{-1}$) and the component groups $A_K(x):=Z_K(x)/Z_K(x)^0$, $x\in\cN_1$ (or $\cN_{-1}$).

\subsection{Young $(k,\pm)$-diagrams}Fix a positive integer $k$. We say that a Young diagram is a $(k,{+})$-diagram (resp. $(k,-)$-diagram) if the rows of the Young diagram are filled with consecutive {\em decreasing} (resp. {\em increasing}) numbers from $\{1,\ldots,k\}$, where $k$ is identified with $0$.  Two Young $(k,\pm)$-diagrams are regarded equivalent if they can be obtained from each other by interchanging rows. We write 
\beq\label{yda}
\lambda^{\pm}=(\lambda_1)_1^{p_1^1}\cdots(\lambda_1)^{p_1^{k}}_k\cdots (\lambda_s)_1^{p_s^1}\cdots(\lambda_s)^{p_s^{k}}_{k}
\eeq
for the Young $(k,\pm)$-diagram such that $\lambda=(\lambda_1)^{p_1^1+\ldots +p_1^k}\cdots (\lambda_s)^{p_s^1+\ldots +p_s^k}$ is the underlying partition with $\lambda_1>\lambda_2>\cdots\lambda_s>0$, and there are $p_i^j$ rows of length $\lambda_i$ that start with a beginning box $j$. We write $|\lambda^{\pm}|=\sum_{i=1}^s(\sum_{j=1}^kp_i^j)\lambda_i$ and
\beqn
\mathbf{d}(\lambda^\pm)=(d_1(\lambda^\pm),\ldots,d_k(\lambda^\pm))
\eeqn
where $d_i(\lambda^\pm)$ denotes the number of boxes filled with $i$ in $\lambda^\pm$. 
 
We will make use of Young $(k,+)$ (resp. $(k,-)$)-diagrams to parametrize nilpotent orbits in $\cN_1$ (resp. $\cN_{-1}$). Fix a dimension vector $\mathbf{d}=(d_1,\ldots,d_k)$. Let us write $\Sigma_{k,\mathbf{d}}^{\pm}$ for the set of Young $(k,\pm)$-diagrams consisting of $\lambda^\pm$ such that $\mathbf{d}(\lambda^\pm)=\mathbf{d}$. Let $\mathbf{1}_k=(1,\ldots,1)\in\bN^k$. We will often omit the superscript $\pm$ when it is clear from the context.

\subsection{Nilpotent orbits and component groups of centralisers}The nilpotent $K$-orbits in $\cN_{\pm 1}$ are parametrized by 

{\bf Type AI.} The set $\Sigma_{m,\mathbf{d}}^{\pm}$

{\bf Type AII.} The subset $\Sigma_{m_0,\bfd}^{A,\pm}\subset \Sigma_{m_0,\bfd}^{\pm}$ consisting of Young $(m_0,\pm)$-diagrams of the form~\eqref{yda} such that
\ber\label{orbit-1}
&p_i^a=p_i^b\text{ if }a+b\equiv\lambda_i\nmod m_0,\ p_i^a\text{ is even if }2a\equiv\lambda_i\nmod m_0,\ i\in[1,s].
\eer

{\bf Type CII.} The subset $\Sigma_{m,\bfd}^{C,\pm}\subset \Sigma_{m,\bfd}^{\pm}$ consisting of Young $(m,\pm)$-diagrams of the form~\eqref{yda} such that
\ber\label{orbit-2}
&p_i^a=p_i^b\text{ if }a+b\equiv\lambda_i-1\nmod m,\ p_i^a\text{ is even if }2a\equiv\lambda_i-1\nmod m,\ i\in[1,s].
\eer 

{\bf Type DII.} The subset $\Sigma_{m,\bfd}^{D,\pm}\subset \Sigma_{m,\bfd}^{\pm}$ consisting of Young $(m,\pm)$-diagrams of the form~\eqref{yda} such that
\ber\label{orbit-3}
&p_i^a=p_i^b\text{ if }a+b\equiv\lambda_i\nmod m,\ p_i^a\text{ is even if }2a\equiv\lambda_i\nmod m,\ i\in[1,s].
\eer

Given a Young diagram $\lambda$, we write $\cO_\lambda$ for the corresponding nilpotent $K$-orbit. Let $x\in\cO_\lambda$. We have \ber
\label{comp-1}&A_K(x)=\bZ/d_\lambda\bZ,\ d_\lambda:=\on{gcd}(\lambda_1,\ldots,\lambda_s)&\text{type AI}\\
\label{comp-2}&A_K(x)=1&\text{type ACDII}.
\eer

The above statements for type AI is well-known. In the remainder of this subsection, we explain the parametrization of nilpotent orbits in type II and prove~\eqref{comp-2}. These are probably known to the experts. For completeness and to fix notations, we include them here. For $x\in\cO_\lambda$, let $\phi_x=(x,y,h)$ be a normal  $\mathfrak{sl}_2$-triple associated to $x$, where $h\in\Lg_0$. We work with $\cN_1$. The case of $\cN_{-1}$ is entirely similar.

Let $x\in\cN_1$. Suppose that $x$ has a Jordan block $v_a\xrightarrow{x} v_{a-1}\xrightarrow{x}\cdots\xrightarrow{x} v_{a-p+1}\xrightarrow{x}0$ of size $p$, where $v_a\in M_{a}$. Let $W_p=\on{span}\{v_{a-i},i\in[0,p-1]\}$. We show that $(\,,\,)|_{W_p}=0$. We have $(x^{p-1}v_a,x^jv_a)=(-1)^{p-1}c_p(v_a,x^{p-1+j}v_a)=0$ if $j\geq 1$, where $c_p=\zeta_m^{p-1}$ in type A, and $c_p=1$ in type CD.  It remains to show that $(x^{p-1}v_a,v_a)=0$.\\
 (Type AII)  
 Suppose first that $p=2k$. We have
$
(x^{2k-1}v_a,v_a)=(-\xi_m)^{k-1}(x^kv_a,x^{k-1}v_a).
$
If $2a-2k\not\equiv0\,\nmod m_0$, then $(x^kv_a,x^{k-1}v_a)=0$. If $2a-2k\equiv0\,\nmod m_0$, then $x^{k-1}v_a\in M_1$ and $(x^kv_a,x^{k-1}v_a)=-(x^kv_a,x^{k-1}v_a)$ implies that $(x^kv_a,x^{k-1}v_a)=0$. Suppose next that $p=2k+1$. 
We have
$
(x^{2k}v_a,v_a)=(-\xi_m)^{k}(x^kv_a,x^{k}v_a).$ 
If $2a-2k\not\equiv1\,\nmod m_0$, then $(x^kv_a,x^{k}v_a)=0$; if $2a-2k\equiv1\,\nmod m_0$, then $x_kv_a\in M_{l+1}$ and $(x^kv_a,x^{k}v_a)=0$ as $(,)|_{M_{l+1}}$ is symplectic.\\ 
(Type CII) 
 If $p$ is even, since $2a-p+1\not\equiv0\,\nmod m$, 
$
( x^{p-1}v_a,v_a)=0.
$
If $p$ is odd, then  $( x^{p-1}v_a,v_a)=(-1)^{\frac{p-1}{2}}( x^{\frac{p-1}{2}}v_a,x^{\frac{p-1}{2}}v_a)=0$.\\ 
(Type DII) If $p$ is even, then $(x^{p-1}v_a,v_a)=(-1)^{p-1}(v_a,x^{p-1}v_a)$ implies that $(x^{p-1}v_a,v_a)=0$. If $p$ is odd, $(x^{p-1}v_a,v_a)=(-1)^{(p-1)/2}(x^{(p-1)/2}v_a,x^{(p-1)/2}v_a)=0$, since $2a-p+1\not\equiv1\,\nmod m$.

It follows that all Jordan blocks appear in pairs. More precisely, if $x$ has a Jordan block $v_a\xrightarrow{x} v_{a-1}\xrightarrow{x}\cdots\xrightarrow{x} v_{a-p+1}\xrightarrow{x}0$ of size $p$, where $v_a\in M_{a}$, then $x$ has another Jordan block $v_b\xrightarrow{x} v_{b-1}\xrightarrow{x}\cdots\xrightarrow{x} v_{b-p+1}\xrightarrow{x}0$ of size $p$, where $v_b\in M_b$ and $a+b-p\equiv 0\,\nmod m_0$ (resp. $a+b-p+1\equiv0\,\nmod m$, $a+b-p\equiv0\,\nmod m$) in type AII (resp. CII, DII). Thus~\eqref{orbit-1},~\eqref{orbit-2} and~\eqref{orbit-3} hold.

 Let $x\in\cO_\lambda$, $\lambda\in\Sigma_{m_0}^A$.  We have $G^{\phi_x}\cong S(\prod_{j=1}^sGL_{\sum_{i=1}^dp_j^i})$. Note that $(x^{k-1}v,w)=-(x^{k-1}w,v)$ if $v,w\in M_a$ and $2a\equiv k\,\nmod m_0$. It follows that $\theta|_{G^{\phi_x}}$ is of type II.

 Let $x\in\cO_\lambda$, $\lambda\in\Sigma_m^C$.  We have $G^{\phi_x}\cong (\prod_{\lambda_j\text{ odd }}Sp_{\sum_{i=1}^dp_j^i})\times(\prod_{\lambda_j\text{ even}}O_{\sum_{i=1}^dp_j^i})$.  Since $2a\not\equiv\lambda_i-1\,\nmod m$ for even $\lambda_i$, $\theta|_{O_{\sum_{i=1}^dp_j^i}}$ is of type II.

 Let $x\in\cO_\lambda$, $\lambda\in\Sigma_m^D$. We have $G^{\phi_x}\cong (\prod_{\lambda_j\text{ even}}Sp_{\sum_{i=1}^dp_j^i})\times(\prod_{\lambda_j\text{ odd}}O_{\sum_{i=1}^dp_j^i})$.  Since $2a\not\equiv\lambda_i\,\nmod m$ for odd $\lambda_i$, $\theta|_{O_{\sum_{i=1}^dp_j^i}}$ is of type II. 
 
 We conclude that $A_K(x)\cong K^{\phi_x}/(K^{\phi_x})^0=1$.

\subsection{Number of nilpotent orbits in type II}Recall $m_0=2l+1$ and $m=2l$. Let 
\beqn
\Sigma_{m_0,2n}^A=\bigsqcup_{|\bfd|=2n}\Sigma_{m_0,\bfd}^A,\ \ \Sigma_{m,2n}^C=\bigsqcup_{|\bfd|=2n}\Sigma_{m,\bfd}^C,\ \  \Sigma_{m,2n}^D=\bigsqcup_{|\bfd|=2n}\Sigma_{m,\bfd}^D.
\eeqn
\begin{lemma}\label{lem-nb1}We have
\bern
&&\sum_n|\Sigma_{m_0,2n}^A|x^n=\prod_{k\geq 1}\frac{1}{(1-x^{k})^{l+1}},\ \ \sum_{n\geq 0}|\Sigma_{m,2n}^C|x^n=\prod_{k\geq 1}\frac{1+x^k}{(1-x^{k})^{l}}\\
&&\sum_{n\geq 0}|\Sigma_{m,2n}^D|x^n=\prod_{k\geq 1}\frac{1}{(1-x^{k})^{l+1}(1+x^k)}.
\eern
\end{lemma}
\begin{proof}
We can count $|\Sigma_{m_0,2n}^A|$ as follows. For a partition $\lambda=(\lambda_1)^{k_1}\cdots(\lambda_s)^{k_s}$ of $n$, we associate weight $\omega_\lambda^A=\prod_{i=1}^s{k_i+l\choose l}$. Then we have $|\Sigma_{m_0,2n}^A|=\sum_{\lambda\in\cP(n)}\omega_\lambda^A$. Note that $\sum_{k=0}^\infty{k+l\choose l}t^k=\frac{1}{(1-t)^{l+1}}$. Thus
\beqn
\sum_n |\Sigma_{m_0,2n}^A|x^n=\prod_{k\geq 1}\sum_{j=0}^\infty x^{kj}{j+l\choose l}=\prod_{k\geq 1}\frac{1}{(1-x^k)^{l+1}}.
\eeqn
Similarly, for a partition $\lambda=(\lambda_1)^{k_1}\cdots(\lambda_s)^{k_s}$ of $n$, we associate weight $$\omega_\lambda^C=\prod_{\substack{i=1,\ldots,s\\\lambda_i\text{ odd}}}{k_i+l\choose l}\prod_{\substack{i=1,\ldots,s\\\lambda_i\text{ even}}}{k_i+l-1\choose l-1}\text{ (resp. $\omega_\lambda^D=\prod_{\substack{i=1,\ldots,s\\\lambda_i\text{ even}}}{k_i+l\choose l}\prod_{\substack{i=1,\ldots,s\\\lambda_i\text{ odd}}}{k_i+l-1\choose l-1}$)}.$$
We have $|\Sigma_{m,2n}^C|=\sum_{\lambda\in\cP(n)}\omega_\lambda^C$ (resp. $|\Sigma_{m,2n}^D|=\sum_{\lambda\in\cP(n)}\omega_\lambda^D$). Thus
\bern
\sum_n |\Sigma_{m,2n}^C|x^n&=&\prod_{\substack{k\geq 1\\k\text{ odd}}}\sum_{j=0}^\infty x^{kj}{j+l\choose l}\prod_{\substack{k\geq 1\\k\text{ eveb}}}\sum_{j=0}^\infty x^{kj}{j+l-1\choose l-1}\\
&=&\prod_{\substack{k\geq 1\\k\text{ odd}}}\frac{1}{(1-x^k)^{l+1}}\prod_{\substack{k\geq 1\\k\text{ even}}}\frac{1}{(1-x^k)^l}=\prod_{k\geq 1}\frac{1+x^k}{(1-x^k)^l}\\
\sum_n |\Sigma_{m,2n}^D|x^n&=&\prod_{\substack{k\geq 1\\k\text{ even}}}\frac{1}{(1-x^k)^{l+1}}\prod_{\substack{k\geq 1\\k\text{ odd}}}\frac{1}{(1-x^k)^l}=\prod_{k\geq 1}\frac{1}{(1-x^k)^{l+1}(1+x^k)}.
\eern 
\end{proof}

 \subsection{Distinguished nilpotent orbits}\label{ssec-dist} Let $x\in\cN_{-1}$. We say that $x$ is distinguished if $\Lg_1^x:=\{y\in\Lg_1\mid[x,y]=0\}$ consists of nilpotent elements, that is, $\Lg_1^x\subset\cN_1$. The $K$-orbit of $x$ is called a distinguished orbit. Similarly we define distinguished nilpotent orbits in $\cN_1$. In what follows we describe the distinguished nilpotent orbits and define some subsets of these orbits in type AI. 
 
{\bf Type AI.} Let $a\in\bZ_+$ and $d:=\on{gcd}(a,m)$. Let ${}^{0}_a\Sigma_{m,\bfd}\subset\Sigma_{m,\bfd}$ denote the subset consisting of Young diagrams $\lambda$ 
such that $a|d_\lambda$ (see~\eqref{comp-1} for the definition of $d_\lambda$) and such that
\beq\label{eqn-biorb}
\prod_{k=0}^{m/d-1} p_j^{i+kd}=0\text{ for each }j\in[1,s] \text{ and each }i\in[1,d].
\eeq
We write $\displaystyle{{}^{0}_a\Sigma_{m,N}=\bigsqcup_{|\bfd|=N}{}^{0}_a\Sigma_{m,\bfd}}$. Note that ${}^{0}_a\Sigma_N\neq\emptyset$ only if $d<m$.

{\bf Type ACDII.} Let  ${}^0\Sigma_{m_0,\bfd}^A\subset\Sigma_{m_0,\bfd}^A$ (resp. ${}^0\Sigma_{m,\bfd}^C\subset\Sigma_{m,\bfd}^C$, ${}^0\Sigma_{m,\bfd}^D\subset\Sigma_{m,\bfd}^D$) denote the subset consisting of Young $m_0$-diagrams (resp. $m$-diagrams) such that
\bern
\min\{p_i^j,j\in[1,m_0]\}\leq 1\text{ (resp. $\min\{p_i^j,j\in[1,m]\}\leq 1$)},\ i\in[1,s].
\eern
Similarly, we define ${}^0\Sigma_{m_0,2n}^A\subset\Sigma_{m_0,2n}^A$ (resp. ${}^0\Sigma_{m,2n}^C\subset\Sigma_{m,2n}^C$, ${}^0\Sigma_{m,2n}^D\subset\Sigma_{m,2n}^D$).  

The set of distinguished orbits are 
\bern
\{\cO_\lambda\mid\lambda\in{}^{0}_1\Sigma_{m,\bfd}\}\text{ AI},\ \ \{\cO_\lambda\mid\lambda\in{}^0\Sigma_{m_0,\bfd}^A\text{ (resp. ${}^0\Sigma_{m,\bfd}^C,{}^0\Sigma_{m,\bfd}^D$)}\}\text{ AII (resp. CII, DII)}.
\eern
\begin{remark}
The set ${}^{0}_1\Sigma_N$ coincides with the set of  {\em aperiodic} orbits in~\cite{L}.
\end{remark}

\begin{lemma}\label{lem-nb2}{\rm (i)} We have
\beqn
\sum_{N\geq 0} |{}^{0}_a\Sigma_{m,N}|x^{N/a}=\frac{\prod_{k\geq 1}(1-x^{\frac{m}{d}k})^d}{\prod_{k\geq 1}(1-x^{k})^m}.
\eeqn
{\rm (ii)} We have
\bern
&&\sum_{n\geq 0}|{}^0\Sigma_{m_0,2n}^A|x^n=\prod_{k\geq 1}\frac{(1-x^{(2l+1)k})}{(1-x^k)^{l+1}},\ \ 
\sum_{n\geq 0}|{}^0\Sigma_{m,2n}^C|x^n=\prod_{k\geq 1}\frac{(1-x^{2lk})(1+x^k)}{(1-x^k)^{l}}\\
&&\sum_{n\geq 0}|{}^0\Sigma_{m,2n}^D|x^n=\prod_{k\geq 1}\frac{(1-x^{2lk})}{(1-x^k)^{l+1}(1+x^k)}.
\eern
\end{lemma}
\begin{proof}
We have $|{}^{0}_a\Sigma_{m,N}|=\sum_{\lambda\in\cP(N/a)}\omega_\lambda$, where for $\lambda=(\lambda_1)^{k_1}\cdots(\lambda_s)^{k_s}$, $\omega_\lambda=\prod_{i=1}^s\omega_{k_i}$, and $\omega_k$ equals the number of degree $k$ monomials in variables $z_j^i$, $j=1,\ldots,m/d$, $i=1,\ldots,m_0$ such that for each $i\in[1,m_0]$, at least one of $z_j^i$ has degree 0. Since $\omega_k$ equals the coefficient of $t^k$ in $\frac{(1-t^{m/m_0})^m_0}{(1-t)^m}$, the same argument as in the proof of Lemma~\ref{lem-nb1} proves (i).

Similarly for $\lambda=(\lambda_1)^{k_1}\cdots(\lambda_s)^{k_s}$, let $\omega_\lambda^\Delta=\prod_{i=1}^s\omega_{k_i}^{\Delta}$, $\Delta=A,C,D$. where $\omega_k^A$ equals the coefficient of $t^k$ in $
\frac{1-t^{2l+1}}{(1-t)^{l+1}}$, $\omega_{2k+1}^C$ (resp. $\omega_{2k}^D$) equals the coefficient of $t^{2k+1}$ (resp. $t^{2k}$) in $
\frac{1-t^{2l}}{(1-t)^{l+1}}$, $\omega_{2k}^C$ (resp. $\omega_{2k+1}^D$) equals the coefficient of $t^{2k}$ (resp. $t^{2k+1}$) in $
\frac{1-t^{2l}}{(1-t)^{l}}$. We have $|{}^0\Sigma_{m_0,2n}^A|=\sum_{\mu\in\cP(n)}\omega_\mu^A$ and $|{}^0\Sigma_{m,2n}^\Delta|=\sum_{\mu\in\cP(n)}\omega_\mu^\Delta$, $\Delta=C,D$. The same argument as in the proof of Lemma~\ref{lem-nb1} proves (ii).
\end{proof}

\subsection{Dual strata}We write $\underline{\cN_{\pm 1}}$ for the set of $K$-orbits in $\cN_{\pm 1}$. 
Let $\cO\in\underline{\cN_{-1}}$. We define the dual stratum $\widecheck\cO\subset\Lg_1$ entirely similarly as in~\cite[\S3.1]{VX4}. Then the supports of the character sheaves are of the form $\overline{\widecheck\cO}$ for those $\cO$ lie in a subset of $\underline{\cN_{-1}}$, which we will describe later.

Let $\cO_\lambda\subset\cN_{-1}$ be a distinguished nilpotent orbit. Then $\widecheck\cO_\lambda$ is a distinguished nilpotent orbit in $\cN_1$. We define $\widecheck\lambda$ to be the Young diagram such that 
\beq\label{checkmu}
\widecheck\cO_\lambda\cong\cO_{\widecheck\lambda}.
\eeq

Suppose that we are in type AI. Let $\lambda$ be as in~\eqref{yda}. For each $i\in[1,s]$, we define $l_i=\min\{p_i^j,j\in[1,m]\}$.  Suppose that $l_i>0$ for some $i$. We define
$$\mu_\lambda=(\lambda_1)^{p^1_1-l_1}_1\cdots(\lambda_1)^{p_1^m-l_1}_m\cdots(\lambda_s)^{p_s^1-l_s}_1\cdots (\lambda_s)^{p_s^m-l_s}_m.$$   
Let $x\in\widecheck\cO_{\lambda}$. Then we have
\beq\label{eqn-comp}
A_K(x)\cong\bZ/\check d_\lambda\bZ,\ \check d_{\lambda}=\begin{cases}\on{gcd}(m\lambda_i,l_i>0,\ d_{\widecheck\mu_\lambda})&\text{ if $\mu_\lambda\neq\emptyset$}
\\md_\lambda&\text{ if $\mu_\lambda=\emptyset$}.\end{cases} 
\eeq

\section{Character sheaves: inner automorphisms of $SL_N$}\label{sec-cs-sl}

In this section we describe the character sheaves in the case of inner automorphisms of $G=SL_N$, that is, type AI in Section~\ref{ssec-gla}. We follow the approach in~\cite{VX3,VX4}. We will also make use of the notations there.

\subsection{Central characters}Let $Z(G)$ denote the center of $G$. We have that $Z(G)^\theta\cong\bZ/N\bZ$. Recall that we have
\beqn
\cA_K(\cN_{\pm1})=\bigsqcup_{\kappa:Z(G)^\theta\to\bC^*}\cA_K(\cN_{\pm1})_\kappa,\ \ \on{Char}_K(\Lg_1)=\bigsqcup_{\kappa:Z(G)^\theta\to\bC^*}\on{Char}_K(\Lg_1)_\kappa
\eeqn
where the subscript $\kappa$ indicates the action of $Z(G)^\theta$. For a positive integer $k$, we define $\cA_K(\cN_{\pm1})_k$ and $\on{Char}_K(\Lg_1)_k$ entirely similarly as in~\cite[\S2.2]{VX3}, that is, the sheaves where $Z(G)$ acts via an order $k$ character $\psi\in(\widehat{\bZ/N\bZ})_k$. 

As mentioned in the introduction, the set $\on{Char}_K(\Lg_1)_1$ has been determined in~\cite{L}.

\subsection{Supports of character sheaves}

Let $a\in\bZ_+$ and $d:=\on{gcd}(a,m)$. Recall the set ${}^0_a\Sigma_{m,\bfd}$ defined in Section~\ref{ssec-dist}. We define the set $\underline{\cN_{-1}}^{cs,a}$ to be the subset of $\underline{\cN_{-1}}$ consisting of the following nilpotent orbits
\bern
&\cO_{a,\mu}:=\cO_{(\frac{a}{d})^{l}_1\cdots(\frac{a}{d})^{l}_m\sqcup\widecheck\mu},\,0\leq l\leq \frac{Nd}{ma},\, \mu\in{}^{0}_{a}\Sigma_{m,\bfd-\frac{al}{d}\mathbf{1}_m}&\text{ if $d<m$}\\
&\cO_{a,\emptyset}:=\cO_{(\frac{a}{m})^{l}_1\cdots(\frac{a}{m})^{l}_m},\,l=\frac{N}{a},&\text{ if $d=m$},
\eern
where $\widecheck\mu$ is defined in~\eqref{checkmu}. 
Making use of~\eqref{eqn-comp} and entirely similar argument as in~\cite{VX4}, we conclude that 
\bern
&&\pi_1^K(\widecheck\cO_{a,\mu})\cong\bZ/\check d_{a,\mu}\bZ\times B_{G_{m,1,l}},\ l=\frac{d(N-|\mu|)}{ma},\,\ \check d_{a,\mu}=\begin{cases}\on{gcd}(ma/d,d_{\mu})&\text{ if $\mu\neq\emptyset$}\\ma/d&\text{ if $\mu=\emptyset$}\end{cases}, 
\eern
where by convention $B_{G_{m,1,0}}=\{1\}$. Note that $a|\check d_{a,\mu}$.

\subsection{Nilpotent support character sheaves}To describe the character sheaves, we begin by giving a conjectural description of the set $\on{Char}^\rn_K(\Lg_1)$ of nilpotent support character sheaves. Given an orbit $\cO_\lambda\subset\cN_{\pm1}$ and $\psi\in\widehat{\bZ/d_\lambda\bZ}$, we write $\cE_{\psi}$ for the irreducible $K$-equivariant local system on $\cO_\lambda$ given by $\psi$. Let $\on{Char}^\rn_K(\Lg_1)_a=\on{Char}^\rn_K(\Lg_1)\cap\on{Char}_K(\Lg_1)_a$.
 
\begin{conjecture}\label{conj-nilp}We have
\ber
\label{biorbital}&& \on{Char}^\rn_K(\Lg_1)_a=\{\on{IC}(\cO_{\mu},\cE_{\psi_a})\mid\mu\in{}^0_a\Sigma_{m,\bfd},\,\psi_a\in(\widehat{\bZ/d_\mu\bZ})_a\}.
\eer
\end{conjecture}
\begin{remark}
The above conjecture holds when $m=2$ by~\cite{VX3}. Moreover,~\eqref{biorbital} holds for $a=1$ by~\cite{L} or by~\cite{H}. As pointed out in~\cite{L},  $\on{Char}^\rn_K(\Lg_1)_1$ is exactly the set of perverse sheaves that give rise to canonical bases for an affine Lie algebra of type A.
\end{remark}
We give some evidence of the above conjecture. 

Suppose that $m\nmid N$ and $\Lg_1$ contains a regular nilpotent element of $\Lg$. Let $\cO_{reg}\subset\cN_1$ be the unique $K$-orbit containing a regular nilpotent and let $\psi_N\in(\widehat{\bZ/N\bZ})_N$. 

\begin{lemma}\label{lem-rn} Suppose that $m\nmid N$. Then $\on{IC}(\cO_{reg},\cE_{\psi_N})$ is a cuspidal character sheaf.
\end{lemma}
\begin{proof} Since $m\nmid N$, in view of~\eqref{eqn-comp}, the only $\widecheck\cO_\lambda$'s that afford a central character of order $N$ is $\cO_{reg}$. Moreover, $\on{Char}_{L^\theta}(\Ll_1)_N=\emptyset $ for any $\theta$-stable Levi subgroup $L$ contained in a $\theta$-stable parabolic subgroup. The lemma follows from central character considerations.
\end{proof}
\begin{corollary} Suppose that $\on{gcd}(a,m)=d<m$. Then $\on{IC}(\cO_{a^t_i},\cE_{\psi_a})$, $\psi_a\in(\widehat{\bZ/a\bZ})_a$, is a nilpotent support character sheaf. Moreover, it is cuspidal if and only if $t=1$.
\end{corollary}
\begin{proof} Suppose that $a=k_0m+a_0$, where $a_0\in(0,m)$. We can assume that $i=a_0$. We have $d_i=(k_0+1)t$, $i\in[1,a_0]$, $d_i=k_0t$, $i\in[a_0+1,m]$.  Let $L$ be a $\theta$-stable Levi subgroup contained in a $\theta$-stable parabolic subgroup such that $L\cong S(GL_a^{\times t})$ and $\cO_{\Ll_1}:=\cO_{a_{a_0}}^{\boxtimes t}\subset \Ll_1$. Such an $L$ exists. We have $\pi_1^{L^\theta}(\cO_{\Ll_1})\cong\bZ/a\bZ$. Let $\psi_a\in(\widehat{\bZ/a\bZ})_a$. Then $\on{IC}(\cO_{\Ll_1},\cE_{\psi_a})\in\on{Char}_{L^\theta}(\Ll_1)$ by Lemma~\ref{lem-rn}. We show that the only nilpotent orbit  $\cO$ such that $\cO\cap(\cO_{\Ll_1}\oplus(\Ln_P)_1)\neq\emptyset$ and such that  $\widehat{\pi_1^{K}(\cO)}_a\neq\emptyset$ is $\cO_{a^t_{a_0}}$. It then follows that 
\beqn
\on{Ind}_{\Ll_1\subset\Lp_1}^{\Lg_1}\on{IC}(\cO_{\Ll_1},\cE_{\psi_a})=\on{IC}(\cO_{a^t_i},\cE_{\psi_a})\oplus\cdots
\eeqn
(up to shift) and the corollary follows.

Let $\cO_\lambda$ be an orbit such that $\widehat{\pi_1^{K}(\cO)}_a\neq\emptyset$. Then $\lambda=(a\mu_1)^{m_1}\cdots(a\mu_s)^{m_s}$. We have $a\mu_i=k_0\mu_im+a_0\mu_i$. Suppose that $a_0\mu_i=k_im+a_i$, $a_i\in[0,m)$.  Then we have
\bern
&&(k_0+1)t=d_1(\lambda)\leq\sum_{i\in[1,s],a_i\neq 0}(k_0\mu_i+k_i+1)m_i+\sum_{i\in[1,s],a_i= 0}(k_0\mu_i+k_i)m_i\\
&&\Rightarrow\sum_{i\in[1,s],a_i\neq 0}\mu_im_i+\sum_{i\in[1,s],a_i= 0}\mu_im_i=t\leq \sum_{i\in[1,s],a_i\neq 0}(k_i+1)m_i+\sum_{i\in[1,s],a_i= 0}k_im_i.
\eern
Note that  $\mu_i\geq k_i+1$. Thus we get $a_i>0$ and $\mu_i=k_i+1$ for all $i$. Suppose that $\mu_i\geq 2$ (or equivalently $k_i\geq 1$) for some $i$, then $a_i<a_0$. It follows that
\bern
&&k_0t+t=d_{a_0}(\lambda)\leq\sum_{i\in[1,s],\mu_i\geq 2}(k_0\mu_i+\mu_i-1)m_i+\sum_{i\in[1,s],\mu_i=1}(k_0\mu_i+\mu_i)m_i\\
&&=k_0t+\sum_{i\in[1,s],\mu_i\geq 2}(\mu_i-1)m_i+\sum_{i\in[1,s],\mu_i=1}\mu_im_i.
\eern
This holds only if $\mu_i=1$ for all $i$. We then conclude that $\lambda=a^t_{a_0}$.
\end{proof}
\begin{remark}\label{rmk-nilp}Suppose that 
$
\mu=(a\mu_1)^{p_1^1}_1\cdots(a\mu_1)^{p_1^m}_m\cdots (a\mu_s)^{p_s^1}_1\cdots(a\mu_s)^{p_s^m}_m\in{}^0_a\Sigma.
$
We expect that $\on{IC}(\cO_\mu,\cE_{\psi_a})$ can be obtained by applying parabolic induction to character sheaves of the form $\on{IC}({\boxtimes\cO_{a_i^{p_i^j}}},\cE_{\psi_a})$ on $\theta$-stable Levi subgroups of the form $S(\prod_{i,j}{GL_{ap_i^j}}^{\times \mu_i})$. It will then follow that the sheaves in $\on{Char}^\rn_K(\Lg_1)_a$ can be obtained by parabolic induction from the character sheaf supported on a regular nilpotent orbit in $\on{Char}_{L^\theta}(\Ll_1)$ for a $\theta$-stable Levi subgroup of the form $S(GL_a^{\times N/a})$.
\end{remark}

\subsection{Character sheaves}Let $\tau\in\cP_d(l)$. Recall the irreducible representation $L_\tau$ of $B_{G_{m,1,l}}$ (see Section~\ref{Hecke}). Let $\cT_{\psi_a,\tau}$ denote the $K$-equivariant local system on $\widecheck\cO_{a,\mu}$  corresponding to the representation $\psi_a\boxtimes L_\tau$ of $\pi_1^K(\widecheck\cO_{a,\mu})$ for $\psi_a\in(\widehat{\bZ/\check d_{a,\mu}\bZ})_a$. 
Assume that Conjecture~\ref{conj-nilp} holds. We have the following  explicit description of the character sheaves.
\begin{theorem}\label{cs-sl}
{\rm (i)} Suppose that $m|a$. We have 
\beqn
\on{Char}_K(\Lg_1)_a=\{\on{IC}(\widecheck\cO_{a,\emptyset},\cT_{\psi_{a},\tau})\mid \psi_a\in(\widehat{\bZ/a\bZ})_a,\,\tau\in\cP_m(N/a)\}.
\eeqn
{\rm (ii)} Suppose that $\on{gcd}(a,m)=d<m$ and Conjecture~\ref{conj-nilp} holds. We have
\bern
\on{Char}_K(\Lg_1)_a&=&\{\on{IC}(\widecheck\cO_{a,\mu},\cT_{\psi_{a},\tau})\mid\cO_{a,\mu}\in\underline{\cN_{-1}}^{cs,a},\,\psi_a\in(\widehat{\bZ/\check d_{a,\mu}\bZ})_a,\,\tau\in\cP_d(l)\}.
\eern
\end{theorem}

 To prove the theorem, we begin by defining a bijection between $\cA_K(\cN_{-1})_a$ and the set of sheaves in Theorem~\ref{cs-sl}, for each $a$. 

Let $\cO_\lambda$ be an orbit such that $a|d_\lambda$. Then we have
\beqn
\lambda=(a\mu_1)^{p_1^1}_1\cdots(a\mu_1)^{p_1^m}_m\cdots (a\mu_t)^{p_t^1}_1\cdots(a\mu_t)^{p_t^m}_m.
\eeqn
For $i\in[1,d]$, $k\in[1,t]$, let 
\bern
l_k^i=\on{min}\{p_k^{i+jd},\,j\in[0,m/d-1]\},  q_k^{jd+i}=p_k^{jd+i}-l_k^i,\,j\in[0,m/d-1].
\eern   Let 
\beqn
\begin{gathered}
\nu^i=(\mu_1)^{l_1^i}(\mu_2)^{l_2^i}\cdots(\mu_t)^{l_t^i},\,i\in[1,d];\ \ 
\mu_\lambda=(a\mu_1)^{q_1^1}_1\cdots(a\mu_1)^{q_1^m}_m\cdots (a\mu_t)^{q_t^1}_1\cdots(a\mu_t)^{q_t^m}_m. \end{gathered}
\eeqn
Then $\tau_\lambda:=(\nu^1,\ldots,\nu^d)\in\cP_d(l_\lambda)$, $l_\lambda=\sum_{i=1}^d|\nu^i|$  and $\mu_\lambda\in{}^0_a\Sigma_{m,N-mla/d}$. Note that $\mu_\lambda=\emptyset$ when $d=m$. The maps 
\beqn
\begin{gathered}
\on{IC}(\cO_\lambda,\cE_{\psi_a})\mapsto\on{IC}(\widecheck\cO_{a,\emptyset},\cT_{\psi_a,\tau_\lambda}),\,\psi_a\in(\widehat{\bZ/d_\lambda\bZ})_a\text{ when $d=m$}\\
\on{IC}(\cO_\lambda,\psi_a)\mapsto\on{IC}(\widecheck\cO_{a,\mu_\lambda},\cT_{\psi_a,\tau_\lambda})\text{ when $d<m$}
\end{gathered}
\eeqn
define  the desired bijections.

It remains to show that the sheaves in Theorem~\ref{cs-sl} are indeed character sheaves. This is done in the next two subsections.

\subsection{Nearby cycle sheaves}Suppose that $\bfd=N/m\mathbf{1}_{m}$. In this case the representation $(K,\Lg_1)$ is a stable polar representation, that is, we have $\overline{\Lg_1^{rs}}=\Lg_1$. We have
$
I=Z_K(\fa)/Z_K(\fa)^0\cong\bZ/m\bZ.$ We apply the nearby cycle sheaf construction in~\cite{GVX} and make use of notations there. 

Let $\chi\in\hat I_\iota$, $\iota|m$, and $P_\chi$ the corresponding nearby cycle sheaf.  Let $s\in W_\fa\cong G_{m,1,r}$ be a reflection of the form $s_{ij}^{(k)}$. Then we have
$
Z_G(\fa_s)_{\on{der}}\cong {SL_2^{\times m}}
$ (product of $m$ copies of $SL_2$)
and $\theta|_{Z_G(\fa_s)_{\on{der}}}$ permutes the $m$ factors of $SL_2$.
Let $t\in W_\fa$ be a reflection of the form $\tau_k$. Then we have
$
Z_G(\fa_t)_{\on{der}}\cong SL_m
$
and $\theta|_{Z_G(\fa_t)_{\on{der}}}$ is  (GIT) stable of rank $1$. 
 Apply~\cite{GVX} and the (GIT) stable rank 1 calculation in~\cite[\S6.1]{VX}, we conclude that
\beqn
\fF(P_\chi)\cong\on{IC}(\Lg_1^{rs},\cM_\chi),\ M_\chi\cong \bC_{\chi}\boxtimes\cH^{\iota}(G_{m,1,r})
\eeqn
where  we write $M_\chi$ for the representation of $\pi_1^K(\Lg_1^{rs})\cong\bZ/m\bZ\times B_{G_{m,1,r}}$ that gives rise to the $K$-equivariant local system $\cM_\chi$ on $\Lg_1^{rs}$. It follows that
\beqn
\{\on{IC}(\Lg_1^{rs},\cT_{\psi_\iota,\tau})\mid\psi_\iota\in\widehat{(\bZ/m\bZ)}_{\iota},\tau\in\cP_\iota(N/m)\}\subset\on{Char}_K(\Lg_1)_\iota.
\eeqn
 
Next, we apply the generalised nearby cycle construction in~\cite{VX3} to produce character sheaves supported on $\widecheck\cO:=\widecheck\cO_{k^l_1\cdots k^l_m}$.  Let $\chi\in\widehat{\bZ/km\bZ}$. The following proposition can be derived using similar argument as in the proof of~\cite[(6.1)]{VX3}.
\begin{proposition}We have
\bern
&&\fF P_{\widecheck\cO,\chi}\cong\on{IC}(\widecheck\cO,\cH_{G_{m,1,l}}^{d}\otimes \bC_\chi)\text{ if $\chi\in(\widehat{\bZ/km\bZ})_{kd}$, $d|m$ and $\on{gcd}(k,m/d)=1$}.
\eern
\end{proposition}
Let $a\in\bZ_+$ and $\on{gcd}(m,a)=d$. We write $k=a/d$ and $l=N/mk$. It follows from the above proposition that
\beq\label{char-nby}
\left\{\on{IC}(\widecheck\cO_{a,\emptyset},\cT_{\psi_{a},\tau})\mid \,\psi_a\in(\widehat{\bZ/mk\bZ})_a,\,\tau\in\cP_d(l)\right\}\subset\on{Char}_K(\Lg_1)_a.
\eeq
Thus part (i) of Theorem~\ref{cs-sl} follows.

Now assume that $d<m$. We prove part (ii) of Theorem~\ref{cs-sl} by showing that the sheaves there can be obtained from parabolic induction. To that end,
let $\{e_i^j,j\in[1,d_i]\}$ be a basis of $M_i$. Let $P$ be the $\theta$-stable parabolic subgroup that stabilises the flag $0\subset V_1\subset V$, where $V_1=\on{span}\{e_i^j,\,i\in[1,m],\,j\in[1,lk]\}$. Let $L\cong S(GL_{lkm}\times GL_{N-lkm})$ be the natural $\theta$-stable Levi subgroup. Consider the stratum $\widecheck\cO_{\Ll_1}:=\widecheck\cO_{k^l_1\cdots k^l_m}\boxtimes\cO_{\mu}\subset\Ll_1$, where $\mu\in{}^0_a\Sigma$. We have $\pi_1^{L^\theta}(\Ll_1)\cong\bZ/\check d_{a,\mu}\bZ\times B_{G_{m,1,l}}$. Let $\psi_a\in(\widehat{\bZ/\check d_{a,\mu}\bZ})_a$ and $\tau\in \cP_d(l)$. Let $\cL_{\psi_a,\tau}$ denote the $L^\theta$-equivariant local system on $\widecheck\cO_{\Ll_1}$ given by the irreducible representation $\psi_a\boxtimes L_\tau$ of $\pi_1^{L^\theta}(\Ll_1)$. Then $\on{IC}(\widecheck\cO_{\Ll_1},\cL_{\psi_a,\tau})\in\on{Char}_{L^\theta}(\Ll_1)_a$ by~\eqref{char-nby} and the assumption that Conjecture~\ref{conj-nilp} holds. Consider the map $\pi:K\times^{P_K}(\widecheck\cO_{\Ll_1}+(\Ln_P)_1)\to\Lg_1$. One checks that $\dim K/P_K=\on{dim}(\Ln_P)_1=Nkl-mk^2l^2$, $\on{dim}\widecheck\cO_{\Ll_1}=l(mk^2l-k+1)+\sum_{i=1}^md_i^2-2Nkl+mk^2l^2-c_\mu$ and $\on{dim}\widecheck\cO_{a,\mu}=\sum_{i=1}^md_i^2-c_\mu-lk+l$, where $c_\mu$ denotes the dimension of centraliser of $x_\mu$ in ${G'}^\theta$ for $\theta|_{G'=GL_{N-lkm}}$. Moreover, the fiber $\pi^{-1}(x)$ consists of $1$ element, for $x\in\widecheck\cO_{a,\mu}$. It follows that the image of $\pi$ equals the closure of $\widecheck\cO_{a,\mu}$.  Entirely similar argument as in~\cite{VX4} proves that
\beq\label{eqn-ind1}
\on{Ind}_{\Ll_1\subset\fp_1}^{\Lg_1}\on{IC}(\widecheck\cO_{\Ll_1},\cL_{\psi_a,\tau})=\on{IC}(\widecheck\cO_{a,\mu},\cT_{\psi_{a},\tau})\oplus\cdots
\eeq
It follows that $\on{IC}(\widecheck\cO_{a,\mu},\cT_{\psi_{a},\tau})\in\on{Char}_K(\Lg_1)_a$. This completes the proof of Theorem~\ref{cs-sl}.

\subsection{Cuspidal character sheaves} In this subsection we give a conjectural description of the set of cuspidal character sheaves. 
\begin{conjecture}\label{conj-cusp}
{\rm (i)}Suppose that $m\nmid N$. Then $\on{Char}^{cusp}_K(\Lg_1)\neq \emptyset$ if and only if the grading is N-regular, that is, $\Lg_1$ contains a regular nilpotent element of $\Lg$. In the latter case, we have
\beqn
\on{Char}^{cusp}_K(\Lg_1)=\{\on{IC}(\cO_{reg},\cE_{\psi_N})\mid\cO_{reg}\subset\cN_1\cap\Lg^{reg},\,\psi_N\in\widehat{(\bZ/N\bZ)}_N\}.
\eeqn
{\rm (ii)}Suppose that $m|N$. Then $\on{Char}^{cusp}_K(\Lg_1)\neq \emptyset$ only if $\bfd=N/m\mathbf{1}_m$. In the latter case we have
\bern
&&\on{Char}^{cusp}_K(\Lg_1)=\bigsqcup_{\substack{d|m\\\on{gcd}(N/m,m/d)=1}}\on{Char}^{cusp}_K(\Lg_1)_{\frac{dN}{m}}\\
&&\on{Char}^{cusp}_K(\Lg_1)_{\frac{dN}{m}}
=\{\on{IC}\left(\widecheck\cO_{(\frac{N}{m})_1\cdots(\frac{N}{m})_m},\cT_{\psi_{\frac{dN}{m}},\tau}\right)\mid\psi_{\frac{dN}{m}}\in(\widehat{\bZ/N\bZ})_{\frac{dN}{m}},\ \tau\in\cP_d(1),\ \,i\in[1,d]\}.
 \eern
\end{conjecture}

\begin{remark}
 Note that when $m=N$ in (ii) we are in the (GIT) stable grading case. In that case the conjecture is compatible with~\cite{VX} where all cuspidal sheaves are expected to have full support.
 \end{remark}
 
As an evidence of the above conjecture, we show that the character sheaves in~\eqref{char-nby} can be obtained from parabolic induction of a $\theta$-stable Levi subgroup of the form $S(GL_{mk}^{\times l})$ contained in a $\theta$-stable parabolic subgroup $P$. We have $d_i=lk$, $i\in[1,m]$. Let $e_i^j,\,j\in[1,lk]$ be a basis of $M_i$, $i\in[1,m]$. Let $P$ be the $\theta$-stable parabolic subgroup of $G$ that stablises the flag $0\subset V_1\subset V_2\subset\cdots\subset V_{l-1}\subset V_l=V$, where $V_s=\on{span}\{e_i^j,\,i\in[1,m],\,j\in[1,sk]\}$. Let $W_s=\on{span}\{e_i^j,\,i\in[1,m],\,j\in[(s-1)k+1,sk]\}$, $s\in[1,l]$, and $L\cong S(\prod_{s=1}^lGL_{W_s})$ the $\theta$-stable Levi subgroup of $P$. Consider the stratum $\widecheck\cO_{\Ll_1}:=(\widecheck{\cO}_{k_1k_2\cdots k_m})^{\boxtimes l}$. We have $\pi_1^{L^\theta}(\widecheck\cO_{\Ll_1})\cong\bZ/mk\bZ\times B_{G_{m,1,1}}^{\times l}$. Consider the $L^\theta$-equivariant local system $\cL_{\psi_a,\tau_1,\ldots,\tau_l}$ on $\widecheck\cO_{\Ll_1}$ given by the irreducible representation $\psi_a\boxtimes L_{\tau_1}\boxtimes L_{\tau_2}\boxtimes\cdots\boxtimes L_{\tau_l}$, where $\psi_a\in(\widehat{\bZ/mk\bZ})_a$ and $\tau_i\in\cP_d(1)$. Then $\on{IC}(\widecheck\cO_{\Ll_1},\cL_{\psi_a,\tau_1,\ldots,\tau_l})\in\on{Char}_{L^\theta}(\Ll_1)_a$. Consider the map $\pi:K\times^{P_K}(\widecheck\cO_{\Ll_1}+(\Ln_P)_1)\to\Lg_1$. One checks that $\dim K/P_K=\on{dim}(\Ln_P)_1=mk^2l(l-1)/2$, $\on{dim}\widecheck\cO_{\Ll_1}=l(mk^2-k+1)$, and $\on{dim}\widecheck\cO_{a,\emptyset}=l(mk^2l-k+1)$. Moreover, the fiber $\pi^{-1}(x)\cong S_l$, the symmetric group of $l$ letters, for $x\in\widecheck\cO_{a,\emptyset}$. It follows that the image of $\pi$ equals the closure of $\widecheck\cO_{a,\emptyset}$.  Entirely similar argument as in~\cite{VX4} proves that
\beq\label{eqn-ind2}
\bigoplus_{(\tau_i)\in\cP_d(1)^l}\on{Ind}_{\Ll_1\subset\fp_1}^{\Lg_1}\on{IC}(\widecheck\cO_{\Ll_1},\cL_{\psi_a,\tau_1,\ldots,\tau_l})=\bigoplus_{\tau\in\cP_d(l)}\on{IC}(\widecheck\cO_{a,\emptyset},\cT_{\psi_{a},\tau})\oplus\cdots\text{ (up to shift).}
\eeq

\begin{remark}If the expectation in Remark~\ref{rmk-nilp} holds, then Conjecture~\ref{conj-cusp} follows from~\eqref{eqn-ind1} and~\eqref{eqn-ind2}. 
\end{remark}

\section{Character sheaves: type II}\label{cs-typeII}

In this section we describe the character sheaves in the case of type II graded Lie algebras, see Section~\ref{ssec-gla}. As in the previous section, we follow the approach in~\cite{VX3,VX4} and make use of the notations there.
\subsection{Supports of character sheaves} Recall the set of distinguished orbits defined in Section~\ref{ssec-dist}. Consider the following set $\underline{\cN_{-1}}^{cs}$ of nilpotent $K$-orbits in $\cN_{-1}$
\bern
\cO_{k,\mu}:=\cO_{1^{2k}_11^{2k}_2\cdots 1^{2k}_{m_0}\sqcup\mu},\ \mu\in{}^0\Sigma_{m_0,\bfd-2k\mathbf{1}_{m_0}}^A&\text{  AII}\\
\cO_{k,\mu}:=\cO_{1^{2k}_11^{2k}_2\cdots 1^{2k}_{m}\sqcup\mu},\ \mu\in{}^0\Sigma_{m,\bfd-2k\mathbf{1}_m}^C\text{ (resp. ${}^0\Sigma_{m,\bfd-2k\mathbf{1}_m}^D$)}&\text{  CII (resp. DII)}.
\eern
 One checks that
\bern
\pi_1^K(\widecheck\cO_{k,\mu})\cong B_{G_{m_0,1,k}} \text{ (resp. $B_{G_{m,1,k}}$)  AII (resp.  CDII)}.
\eern

\subsection{Nilpotent support character sheaves}Let us write ${}^0\Sigma$ for ${}^0\Sigma_{m_0,\bfd}^A$, ${}^0\Sigma_{m,\bfd}^{C}$ or ${}^0\Sigma_{m,\bfd}^{D}$. %We have
\begin{conjecture}\label{conj-biorbital}
 We have \beq\label{cs-nilp2}
\on{Char}_K^{\rn}(\Lg_1)=\{\on{IC}(\cO_{\mu},\bC)\mid \mu\in{}^0\Sigma\}.
\eeq

\end{conjecture}
Note that by~\cite[\S3.8]{L}, we have
 $\on{Char}^{\rn}_K(\Lg_1)\subset\{\on{IC}(\cO_\mu,\bC)\mid\mu\in{}^0\Sigma\}$.

Let $B$ be a $\theta$-stable Borel subgroup. By~\cite[3.2(c)]{L}, we have $\Lb_i=(\Ln_B)_i$ for $i\neq 0$ in type CDII. One can check that this holds as well in type AII.  Let $\pi_B:K\times^{B_K}\Lb_1\to\Lg_1$. We have $\on{Im}\pi_B=\overline{\cO_B}$ for some nilpotent orbit $\cO_B\subset\cN_1$. As in~\cite{VX4}, we call the orbits $\cO_B$ Richardson orbits attached to $\theta$-stable Borel groups. 
\begin{remark}\label{rmk-nil2}We expect that all the sheaves in~\eqref{cs-nilp2} can be obtained from parabolic induction $\on{Ind}_{\Lt_1=\{0\}\subset\Lb_1}^{\Lg_1}\delta$ of the skyscraper sheaf $\delta$ on $\Lt_1=\{0\}$, where $\Lt=\on{Lie}T$,  and $T$ is a $\theta$-stable maximal torus contained in a $\theta$-stable Borel subgroup $B$. It is likely that every $\cO_{\mu}$,  $\mu\in{}^0\Sigma$, is a Richardson orbit attached to some $\theta$-stable Borel subgroup, which will then imply the expectation.
\end{remark}

\subsection{Character sheaves}
We give an explicit description of the character sheaves assuming Conjecture~\ref{cs-nilp2} holds.

Let $\rho\in\cP(k)$. Recall the irreducible representation $L_\rho$ of $B_{G_{m,1,k}}$. Let us write $\cL_\rho$ for the $K$-equivariant local system on $\widecheck\cO_{k,\mu}$ given by the irreducible representation $L_\rho$ of $\pi_1^{K}(\widecheck\cO_{k,\mu})$.
\begin{theorem}\label{cs}Suppose that Conjecture~\ref{conj-biorbital} holds. We have
\beq\label{eqn-csII}
\on{Char}_K(\Lg_1)=\{\on{IC}(\widecheck\cO_{k,\mu},\cL_\rho)\mid \cO_{k,\mu}\in \underline{\cN_1}^{cs},\,\rho\in\cP(k)\}.
\eeq
\end{theorem}
Recall $r=\on{min}d_i$. Let $\mu_0=1^{k_1}_1\cdots 1^{k_{m_0}}_{m_0}$ (resp. $1^{k_1}_1\cdots 1^{k_m}_m$), where $k_i=d_i-2r$.
\begin{corollary}We have
\beqn
\on{Char}_K^{\rf}(\Lg_1)=\{\on{IC}(\widecheck\cO_{r,\mu_0},\cL_\rho)\mid \rho\in\cP(r)\}.
\eeqn
\end{corollary}

\subsection{Proof of Theorem~\ref{cs}}Let us write $I_d^A=\{1,\ldots,m_0\}$, $I_m^C=I_m^D=\{1,\ldots,m\}$. Let $\lambda$ be as in~\eqref{orbit-1} (resp.~\eqref{orbit-2},~\eqref{orbit-3}). We define
\beqn
l_i=\min\{[\frac{p_i^a}{2}],\,a\in I_{m_0}^A\text{ (resp. $I_m^C$, $I_m^D$)}\},\ q_i^a=p_i^a-2l_i,\,a\in I_d^A\text{ (resp. $I_m^C$, $I_m^D$)},\ i\in[1,s].
\eeqn
Define
\bern
&\nu=(\lambda_1)^{l_1}\cdots(\lambda_s)^{l_s},\ \mu=(\lambda_1)_1^{q_1^1}\cdots(\lambda_1)^{q_1^{m_0}}_{m_0}\cdots (\lambda_s)_1^{q_s^1}\cdots(\lambda_s)^{q_s^{m_0}}_{m_0}\\
&\text{(resp.  $\mu=(\lambda_1)_1^{q_1^1}\cdots(\lambda_1)^{q_1^{m}}_{m}\cdots (\lambda_s)_1^{q_s^1}\cdots(\lambda_s)^{q_s^{m}}_{m}$)}
\eern
The map
\beqn
\on{IC}(\cO_\lambda,\bC)\mapsto\on{IC}(\widecheck\cO_{|\nu|,\mu},\cL_\nu)
\eeqn
defines a bijection between the set $\cA_K(\cN_{-1})$ of simple $K$-equivariant perverse sheaves on $\cN_{-1}$ and the set of sheaves in Theorem~\ref{cs}. Thus to prove Theorem~\ref{cs}, it suffices to show that the sheaves in~\eqref{eqn-csII} are indeed character sheaves. We prove this for type AII. The other cases are entirely similar and we omit the details. 

We have $d_i=d_{m_0+1-i}$ and $d_{l+1}$ is even. Let $\{e_i^j,\,j\in[1,d_i]\}$ be a basis of $M_i$, $i\in[1,m_0]$ such that $(e_i^j,e_s^t)=\delta_{i+s,m_0+1}\delta_{j,t}$, $i\in[1,l]$,  $(e_{l+1}^j,e_{l+1}^t)=\delta_{j+t,d_{l+1}+1}$, $j\in[1,r]$.

We first consider the case when $d_i=2r$ for all $i$. 
We show that
\beq\label{ind1}
\{\on{IC}(\Lg_1^{rs},\cL_\rho)\mid\rho\in\cP(r)\}\subset\on{Char}_{K}(\Lg_1).
\eeq
 Let $P$ be the $\theta$-stable parabolic subgroup that stabilises the flag $0\subset W_{m_0r}=W_{m_0r}^{\perp}\subset V$, where $W_{m_0r}=\on{span}\{e_i^j,\,i\in[1,m_0],\,j\in[1,r]\}$. Let $L$ be the $\theta$-stable Levi subgroup contained in $P$ such that $L\cong S(GL_{W_{m_0r}}\times GL_{U_{m_0r}})$, where $U_{m_0r}=\on{span}\{e_i^j,\,i\in[1,m_0],\,j\in[r,2r]\}$. Then $(L^\theta,\Ll_1)$ can be identified with $(GL_r^{\times m_0},\on{Hom}(\bC^r,\bC^r)^{\oplus m_0})$. That is, the pair arising from the order $m_0$ inner automorphism of $GL_{m_0r}$ such that $d_i=r$, $i\in[1,m_0]$. Let $\tau\in\cP(r)$. Consider the map $\pi:K\times^{P_K}\Lp_1\to\Lg_1$. One checks readily that $\pi$ is surjective and $\pi^{-1}(x)\cong\bP_1^k$. Consider the sheaf $\on{IC}(\Ll_1^{rs},\cT_{\psi_1,\tau})\in\on{Char}_{L^\theta}(\Ll_1)$. Entirely similar argument as before shows that $\on{Ind}_{\Ll_1\subset\Lp_1}^{\Lg_1}\on{IC}(\Ll_1^{rs},\cT_{\psi_1,\tau})\cong\on{IC}(\Lg_1^{rs},\cL_\tau)\oplus\cdots$ (up to shift). This proves~\eqref{ind1}.

Returning to the general case. We assume that Conjecture~\ref{conj-biorbital} holds. Let $P$ be the $\theta$-stable parabolic subgroup that stabilises the flag $0\subset W_{m_0k}\subset W_{m_0k}^{\perp}\subset V$, where $W_{m_0k}=\on{span}\{e_i^j,\,i\in[1,l+1],\,j\in[1,k],e_i^j,\,i\in[l+2,m_0],\,j\in[k+1,2k]\}$. Let $L$ be the $\theta$-stable Levi subgroup contained in $P$ such that $L\cong S(GL_{W_{m_0k}}\times GL_{U_{m_0k}}\times GL_{E_{2n-2m_0k}})$, where $U_{m_0k}=\on{span}\{e_i^j,\,i\in[1,l],\,j\in[k,2k],e_i^j,\,i\in[l+2,m_0],\,j\in[1,k],\,e_{l+1}^j,\,j\in[a_{l+1}+1-k,a_{l+1}]\}$ and $E_{2n-2m_0k}=\on{span}\{e_i^j,\,i\in[1,l]\cup[l+2,m_0],\,j\in[2k+1,d_i],\,e_{l+1}^j,\,j\in[k+1,d_{l+1}-k]\}$. Let $\mu\in{}^0\Sigma_{2n-m_0k}$. Consider the stratum $\widecheck\cO_{\Ll_1}:=\widecheck\cO_{1^{2k}_1\cdots 1^{2k}_d}\times\cO_\mu$. We have $\pi_1^{L^\theta}(\widecheck\cO_{\Ll_1})\cong B_{G_{m_0,1,k}}$. Let $\tau\in\cP(k)$ and let $\cF_\tau$ denote the $K$-equivariant local system on $\widecheck\cO_{\Ll_1}$ given by the irreducible representation $L_\tau$ of $B_{G_{m_0,1,k}}$. Then $\on{IC}(\widecheck\cO_{\Ll_1},\cF_\tau)\in\on{Char}_{L^\theta}(\Ll_1)$ by~\eqref{ind1} and Conjecture~\ref{conj-biorbital}. Consider the map $\pi:K\times^{P_K}(\overline{\widecheck\cO_{\Ll_1}}+(\Ln_P)_1)\to\Lg_1$. Let $\widecheck\cO=\widecheck\cO_{1^{2k}_1\cdots 1^{2k}_{m_0}\sqcup\mu}$. Let $k_i=d_i-2k$. We have
$\dim\widecheck\cO-\dim\widecheck\cO_{\Ll_1}=m_0k^2+4k\sum_{i=1}^lk_i+2kk_{l+1}-k$, $\dim K/P_K=lk^2+\frac{k^2+k}{2}+2k\sum_{i=1}^lk_i+kk_{l+1}$, and $\dim (\Ln_P)_1=lk^2+\frac{k^2-k}{2}+2k\sum_{i=1}^lk_i+kk_{l+1}.
$ Moreover, for $x\in\widecheck\cO$, $\pi_1^{-1}(x)\cong\bP_1^k$. We conclude that $\on{Im}\pi=\overline{\widecheck\cO}$ and 
\beq\label{eqn-ind3}
\on{Ind}_{\Ll_1\subset\Lp_1}^{\Lg_1}\on{IC}(\widecheck\cO_{\Ll_1},\cF_\tau)\cong\on{IC}(\widecheck\cO_{k,\mu},\cL_\tau)\oplus\cdots\text{ (up to shift).}
\eeq
 This completes the proof of Theorem~\ref{cs} for type AII.
\begin{remark}If the expectation in Remark~\ref{rmk-nil2} holds, then together with~\eqref{eqn-ind3}, it implies that there are no cuspidal character sheaves in type II.\end{remark}

\end{document}